\documentclass[11pt]{amsart}
\usepackage[utf8]{inputenc}
\usepackage[T1]{fontenc}
\usepackage{lmodern}
\usepackage{amsmath,amssymb,amsthm,mathrsfs,mathtools}
\usepackage[colorlinks=true,linkcolor=blue,citecolor=blue,urlcolor=blue]{hyperref}
\usepackage{orcidlink}
\newtheorem{theorem}{Theorem}
\newtheorem{corollary}[theorem]{Corollary}
\newcommand{\G}{\mathscr{G}}
\newcommand{\E}{\mathbb{E}}
\allowdisplaybreaks
\title[Sidorenko's inequality for homomorphism counts]{A very short proof of Sidorenko's inequality for counts of homomorphisms between graphs}
\author[L. L\"uchtrath]{Lukas L\"uchtrath\,\orcidlink{0000-0003-4969-806X}}
\address{Weierstrass Institute for Applied Analysis and Stochastics, Mohrenstr.\ 39, 10117 Berlin, Germany}
\email{lukas.luechtrath@wias-berlin.de}
\author[C. M\"onch]{Christian M\"onch\,\orcidlink{0000-0002-6531-6482}}
\address{Johannes Gutenberg-Universit\"at Mainz,\newline Staudingerweg 9, 55128 Mainz, Germany}
\email{cmoench@uni-mainz.de}
\date{November 25, 2026}
\begin{document}
\begin{abstract}
A fundamental extremality result due to Sidorenko (Discrete Math.\ 131.1-3, 1994) states that among all connected graphs $G$ on $k$ vertices, the $k$-vertex star maximises the number of graph homomorphisms of $G$ into any graph $H$. We provide a new short proof of this result using only a simple recursive counting argument for trees and H\"older's inequality.
\end{abstract}
\maketitle
Let $\sharp \operatorname{hom}(G,H)$ denote the number of graph homomorphisms from a graph $G=(V(G),E(G))$ into an \emph{image graph} $H=(V(H),E(H))$, that is,
\[
\operatorname{hom}(G,H)=\{f:V(G)\to V(H)\mid (u,v)\in E(G)
\Rightarrow (f(u),f(v))\in E(H)\}.
\] The following general inequality was proven by Sidorenko~\cite{Sidorenko_1994}:
\begin{theorem}\label{thm:sidoalt}
Let $G$ denote any connected graph on $k+1$ vertices and $S_k$ the star graph with $k$ edges. Then
\begin{equation*}
\sharp \operatorname{hom}(G,H)\leq \sharp \operatorname{hom}(S_k,H) \text{ for any graph }H.
\end{equation*}
\end{theorem}

%Why do we care?
Define a partial order $\preceq$ on the set $\G_k$ of connected graphs on $k+1$ vertices by declaring $G'\succeq G$, if \[\sharp \operatorname{hom}(G',H)\geq \sharp \operatorname{hom}(G,H) \text{ for every graph }H.\] Aside from explicitly classifying the maximal elements of $\preceq$ and its usefulness for counting problems in graphs, Sidorenko's inequality is a valuable tool in various applications. For instance, if $A$ is any positive symmetric $n\times n$ matrix, Theorem~\ref{thm:sidoalt} implies the inequality
\[
\sum_{i,j\leq n}A^k_{i,j}\leq \sum_{i=1}^n\Big(\sum_{j=1}^n A_{i,j}\Big)^k,
\]
which is due to Hoffman~\cite{Hoffman67}. Further applications can be found in the context of noise sensitivity of Boolean functions \cite{PSW16} and in the theory of complex networks, where Theorem~\ref{thm:sidoalt} is central in deriving moment criteria for the continuity of clustering coefficients under local weak convergence \cite{Kurau22}.
%What was known before?

Three proofs of Theorem~\ref{thm:sidoalt} can be found in the literature. Sidorenko's original proof relies on the remarkable fact that the relation \(G\preceq G'\) is in one-to-one correspondence to an ordering of integral functionals on measure spaces which can be associated with $G$ and $G'$. Certain combinatorial operations on graphs map to $\log$-convex combinations of these functionals for which Hölder-type inequalities hold. The corresponding inequalities for the homomorphism counts are then used to establish the extremality of star graphs and further relations between the homomorphism counts of concrete examples of graphs.

Csikv\'ari and Lin~\cite{csikvari_graph_2014} provide another proof of Theorem~\ref{thm:sidoalt} that is close in spirit to Sidorenko's work but uses the Wiener index (the sum of all distances between pairs of vertices in a graph) and more elementary combinatorial operations on graphs to conclude.

Finally, Levin and Peres~\cite{levin_counting_2017} prove Theorem~\ref{thm:sidoalt} by a brief and elegant probabilistic argument that connects the homomorphism count to the stationary distribution of the simple random walk on the target graph.

The aim of this note is to present a new, remarkably elementary proof of Sidorenko's bound that relies solely on a short recursive enumeration argument and Hölder's inequality on finite probability spaces.

\begin{proof}[Proof of Theorem~\ref{thm:sidoalt}]
 Fix an arbitrary image graph $H=(V(H),E(H))$. For $k\leq 2$, deleting edges reduces the assertion to equality for a star. We may therefore assume that $k\geq 3$ and that $H$ has an edge, since for an edgeless $H$ both homomorphism counts vanish. Observe that removing edges from $G$ can only increase  $\sharp \operatorname{hom}(G,H)$, hence it suffices to show
\begin{equation}\label{eq:targetineq}
\sharp \operatorname{hom}(T,H)\leq \sharp \operatorname{hom}(S_k,H),
\end{equation}
whenever $T$ is any $k$-edge tree. Let $\mathcal{T}(k,\ell)$ denote the set of $k$-edge trees with precisely $\ell\leq k$ leaves. In particular, we have $\mathcal{T}(k,k)=\{S_k\}$. The bound \eqref{eq:targetineq} follows, if we can show that
\begin{equation}\label{eq:Cor28}
\max_{T\in \mathcal{T}(k,\ell)} \sharp \operatorname{hom}(T,H) \leq \max_{T\in \mathcal{T}(k,\ell+1)} \sharp \operatorname{hom}(T,H) \text{ for any }\ell<k.
\end{equation}
To this end, we demonstrate that for every non-star $T\in \mathcal{T}(k,\ell)$ there exists some $k$-edge tree $T'$ with one more leaf that admits at least the same number of homomorphisms into $H$ as $T$. Denote by $\operatorname{sk}(T)$, the \emph{skeleton tree} of $T$, obtained by removing all leaves from $T$. Since $\ell<k$, $\operatorname{sk}(T)$ has at least $2$ leaves. We designate two leaves $b_1,b_2$ of $\operatorname{sk}(T)$ and denote by $T(b_1,b_2)$ the graph obtained from $T$ by removing all leaves adjacent to $b_1$ and $b_2$. We write \smash{$\vec{d}_1,\vec{d}_2$} for the number of leaves removed at $b_1$ and $b_2$, respectively. Write
\[
N(u,v)=\sharp\{f\in\operatorname{hom}(T(b_1,b_2),H): f(b_1)=u,\ f(b_2)=v\}.
\]
Calculating $\sharp \operatorname{hom}(T,H)$ by first counting all maps of $T(b_1,b_2)$ and then the possible choices for the images of the remaining leaves yields
\begin{equation}\label{eq:shortcut}
\begin{aligned}
\sharp \operatorname{hom}(T,H) & = \sum_{u,v\in V(H)} N(u,v) \operatorname{deg}(u)^{\vec{d}_1}\operatorname{deg}(v)^{\vec{d}_2} \\ &
=\sharp\operatorname{hom}(T(b_1,b_2),H) \sum_{u,v\in V(H)}p(u,v) \operatorname{deg}(u)^{\vec{d}_1}\operatorname{deg}(v)^{\vec{d}_2},
\end{aligned}
\end{equation}
where $p(u,v)$ denotes the probability that a uniformly chosen homomorphism from $T(b_1,b_2)$ into $H$ maps $b_1$ to $u$ and $b_2$ to $v$. Denoting the marginals of $p(\cdot,\cdot)$ by $p_1(\cdot)$ and $p_2(\cdot)$, we conclude with the help of Hölder's inequality
\begin{equation}\label{eq:finaleq}
\begin{aligned}
& \sharp \operatorname{hom}(T,H) \\
& \leq \sharp\operatorname{hom}(T(b_1,b_2),H)
\Bigg(\sum_{u\in V(H)}p_1(u)\operatorname{deg}(u)^{\vec{d}_1+\vec{d}_2}\Bigg)^{\frac{\vec{d}_1}{\vec{d}_1+\vec{d}_2}}\\
&\qquad\qquad\cdot
\Bigg(\sum_{u\in V(H)}p_2(u)\operatorname{deg}(u)^{\vec{d}_1+\vec{d}_2}\Bigg)^{\frac{\vec{d}_2}{\vec{d}_1+\vec{d}_2}}.
\end{aligned}
\end{equation}
We assume, without loss of generality, that
\[
\sum_{u\in V(H)}p_1(u)\operatorname{deg}(u)^{\vec{d}_1+\vec{d}_2}\geq \sum_{u\in V(H)}p_2(u)\operatorname{deg}(u)^{\vec{d}_1+\vec{d}_2},
\]
since if the opposite inequality holds, we may reverse the choice of $b_1,b_2$ at the beginning. Thus, \eqref{eq:finaleq} leads to
\begin{equation}\label{eq:forkoff}
\begin{aligned}
\sharp \operatorname{hom}(T,H) & \leq \sharp\operatorname{hom}(T(b_1,b_2),H) \sum_{u\in V(H)}p_1(u)\operatorname{deg}(u)^{\vec{d}_1+\vec{d}_2}\\
& = \sum_{u\in V(H)}\sharp \{f\in \operatorname{hom}(T(b_1,b_2),H): f(b_1)=u \} \operatorname{deg}(u)^{\vec{d}_1+\vec{d}_2}.
\end{aligned}
\end{equation}
The last expression equals $\sharp \operatorname{hom}(T',H)$, where $T'\in \mathcal{T}(k,\ell+1)$ is obtained from $T(b_1,b_2)$ by attaching $\vec{d}_1+\vec{d}_2$ leaves to $b_1$ and consequently has precisely one more leaf than $T$.
\end{proof}
In fact, the above line of reasoning also establishes a slightly stronger form of the statement that coincides with Sidorenko's original formulation of the result with only a minor refinement.
\begin{corollary}[{\cite[Theorem 1.2]{Sidorenko_1994}}]\label{cor:thm1.2}
Let $k\geq 3$ and let $T$ denote any non-star $k$-edge tree, $S_k$ the star graph with $k$ edges and $T_{k-1,1}$ any tree obtained by attaching a single leaf to a leaf of $S_{k-1}$. Then
\(
T\preceq T_{k-1,1} \preceq S_k.
\)
\end{corollary}
\begin{proof}
Fix an image graph $H$ with at least one edge; for an edgeless $H$ the assertion is immediate. The leaf-increasing operation in the proof of Theorem~\ref{thm:sidoalt}, applied repeatedly for this fixed $H$, transforms any non-star $k$-edge tree into a tree with $k-1$ leaves without decreasing its homomorphism count into $H$. It therefore suffices to maximise this count over $\mathcal{T}(k,k-1)$.

All graphs in $\mathcal{T}(k,k-1)$ have $K_2$ as their skeleton tree. Thus, $T(b_1,b_2)=K_2$, and the distribution $p(u,v)$ in \eqref{eq:shortcut} is uniform on the oriented edges of $H$, independently of the leaf counts. In particular, both endpoint degrees are positive, so the logarithms below are well-defined on the support of $p$. Since the latter graph is symmetric under swapping $b_1$ and $b_2$, it follows that the distribution $p(\cdot,\cdot)$ used in \eqref{eq:shortcut} is symmetric. Let $(U,V)$ denote a pair of random variables with distribution $p(\cdot,\cdot)$. Set
\[
g(v)=(\vec{d}_1+\vec{d}_2)\log(\operatorname{deg}(v)),
\qquad p_1=\frac{\vec{d}_1}{\vec{d}_1+\vec{d}_2}.
\]
Then
\[
\begin{aligned}
\sum_{u,v\in V(H)}p(u,v)\operatorname{deg}(u)^{\vec{d}_1}\operatorname{deg}(v)^{\vec{d}_2}
&=\E\big[\textup{e}^{p_1g(U)+(1-p_1)g(V)}\big].
\end{aligned}
\]
Since $p(u,v)=p(v,u)$, we have that the map
$\phi:[0,1]\to[0,\infty)$ given by
\[
	\phi(p)=\E\big[\textup{e}^{p g(U)+(1-p)g(V)}\big]
\]
is symmetric. The proof of Theorem~\ref{thm:sidoalt} shows that $\phi$ attains its maxima at $0$ and $1$. Furthermore,
\[
\begin{aligned}
\phi'(p)&=\E\big[(g(U)-g(V))\textup{e}^{p g(U)+(1-p)g(V)}\big],\\
\phi''(p)&=\E\big[(g(U)-g(V))^2\textup{e}^{p g(U)+(1-p)g(V)}\big],
\end{aligned}
\]
hence $\phi$ is convex and attains its minimum at $p=1/2$. Consequently, $\phi$ is non-decreasing on $[1/2,1]$, which implies that $\mathcal{T}(k,k-1)$ is totally ordered with respect to $\preceq$, where the minimum is attained at $\vec{d}_1=\vec{d}_2=(k-1)/2$ if $k$ is odd and at $\vec{d}_1=k/2, \vec{d}_2=k/2-1$ if $k$ is even, and the maximum is attained at $\vec{d}_1=k-2, \vec{d}_2=1$. This maximum corresponds to $T_{k-1,1}$. Combining this comparison with the preceding reduction proves the first inequality for the arbitrary image graph $H$; the second follows from Theorem~\ref{thm:sidoalt}.
\end{proof}
We conclude this note by remarking that one can replace the use of H\"older's inequality in the proof of Theorem~\ref{thm:sidoalt} by an equally short argument using the weighted AM-GM inequality. It is an elementary analytic exercise to see that the two inequalities are in fact equivalent, so both variants of our argument are equally fundamental.
\begin{samepage}\begin{proof}[Alternative proof of Theorem~\ref{thm:sidoalt} without H\"older's inequality]
Employing the notation from the proof of Corollary~\ref{cor:thm1.2}, now with $(U,V)$ distributed according to $p(u,v)$ in \eqref{eq:shortcut} without requiring symmetry, we see that
\[
\phi(p)\leq p\E\big[\textup{e}^{g(U)}\big]+(1-p)\E\big[\textup{e}^{g(V)}\big]\leq \max\{\E\big[\textup{e}^{g(U)}\big],\E\big[\textup{e}^{g(V)}\big]\},\quad p\in[0,1]
\]
where the first inequality follows from applying the weighted AM-GM inequality. After interchanging $b_1,b_2$ if necessary, multiply this inequality by the number of homomorphisms from $T(b_1,b_2)$ into $H$ to recover the bound \eqref{eq:forkoff}.
\end{proof}\end{samepage}
%\pagebreak

\noindent\textbf{\large Funding acknowledgement.} LL received support from the Leibniz Association within the Leibniz Junior Research Group on \textit{Probabilistic Methods for Dynamic Communication Networks} as part of the Leibniz Competition (grant no.\ J105/2020). CM's research is funded by Deutsche Forschungsgemeinschaft (DFG, German Research Foundation) – SPP 2265 443916008.

\bibliographystyle{amsplain}
\bibliography{bib}
\end{document}